\documentclass[a4paper,12pt]{amsart}
\usepackage{amsmath,amssymb}
\usepackage[english]{babel}
\usepackage{graphicx}
\usepackage{caption}
\newtheorem{thm}{Theorem}
\newtheorem{lemma}[thm]{Lemma}
\newtheorem{prop}[thm]{Proposition}
\newtheorem{cor}[thm]{Corollary}
\newtheorem{exam}[thm]{Example}
\newtheorem{rem}[thm]{Remark}

\def\g{\gamma}
\def\ds{\displaystyle}
\begin{document}

\title{Geometric optimization problems
 generated by plane curves}

\author{Petar Kenderov}
\address{P. Kenderov\\Institute of Mathematics and Informatics\\Bulgarian Academy
of Sciences\\Acad. G. Bonchev 8, 1113 Sofia, Bulgaria}
\email{vorednek@gmail.com}

\author{Oleg Mushkarov}
\address{O. Mushkarov\\Institute of Mathematics and Informatics\\Bulgarian Academy
of Sciences\\Acad. G. Bonchev 8, 1113 Sofia, Bulgaria}
\email{muskarov@math.bas.bg}

\author{Nikolai Nikolov}
\address{N. Nikolov\\Institute of Mathematics and Informatics\\Bulgarian Academy
of Sciences\\Acad. G. Bonchev 8, 1113 Sofia, Bulgaria \vspace{1mm}
\newline Faculty of Information Sciences\\
State University of Library Studies and Information Technologies\\
Shipchenski prohod 69A, 1574 Sofia, Bulgaria}
\email{nik@math.bas.bg}

\thanks{The first-named author was supported by the Bulgarian
National Science Fund,
 under Grant No. KP-06-H92/6 (December 8, 2025).}

\thanks{The third-named author was partially supported by the Bulgarian
National Science Fund, Ministry of Education and Science of Bulgaria, under contract KP-06-N82/6.}



\pagenumbering{arabic}

\begin{abstract}
Let $\g_1$ and $ \g_2$ be regular $C^1$-smooth curves in the plane
and $\g$ be a regular $C^2$-smooth curve in the same plane.
Consider all triples of points
$(A, A_1, A_2)$, $A\in \g$, $A_1\in \g_1$, $A_2\in \g_2$,  such that $A_1\neq A_2$, and the line  $A_1 A_2$ is the normal to $\g$ at $A$.
We show that, if $\g$ has non-vanishing curvature and the triple
$(A^0, A_1^0,A_2^0 )$ is a local
maximum or a local minimum for   the distance
$|A_1A_2|$ between the points $A_1$ and
$A_2$, then the following three lines either meet at a single point or are parallel: the normal to
$\g_1$ at $A_1^0$, the normal to $\g_2$ at $A_2^0$ and the line
which is perpendicular to $A_1^0A_2^0$, and passing through the center of curvature of $\g$ at $A^0$.

The particular case of this optimization problem,
when $\g$ is a circle
with a given center  $O$, coincides with the already partially studied  problem of finding the locally shortest (or the locally longest) non-degenerate segments
$[A_1A_2]$ such that $A_1\in\g_1$, $A_2\in\g_2$
 and $O\in A_1A_2$. We also show that the seemingly
different problem of finding the locally shortest (or the locally longest)
non-degenerate segments $[A_1A_2]$,  such that
 $A_1 \in \g_1$, $A_2 \in \g_2$, and the
 line $A_1A_2$ is tangent to $\g$ is also, in essence, a particular case of the above optimization problem.
 We consider in detail the ``degenerate cases'' naturally appearing in this setting
(when, for instance,  $\g_1$ or $\g_2$ coincide with $\g$, or when the optimal line $A_1^0A_2^0$ is tangent to at least one of
$\g_1$ or $\g_2$).

\end{abstract}

\maketitle

\section{Introduction}\label{si}

Let $\g_1$ and $ \g_2$ be  regular $C^1$-smooth plane curves
and $\g$ be a regular $C^2$-smooth curve in the same plane.
We call a triple of points $(A, A_1, A_2)$ {\it admissible} if $A\in \g$,
$A_1\in \g_1$, $A_2\in \g_2$,
$A_1\neq A_2$,
and the line  $A_1 A_2$ is the normal to $\g$ at  $A$.
We say that an admissible triple  $(A, A_1, A_2)$,
 is a point of {\it local maximum} ({\it local minimum})
if there exist open sets $U \ni A$,
 $U_1 \ni A_1$,
$U_2 \ni A_2$ such that for every admissible triple
$(A', A'_1, A'_2)$ with $A' \in U'$,
 $A'_1 \in U_1$,
$A'_2 \in U_2$, the following inequality holds:
$|A_1A_2| \geq |A'_1A'_2|$ ($|A_1A_2| \leq |A'_1A'_2|$). The natural optimization problem that appears here is

\medskip

{\it  Find necessary conditions for a given admissible
triple $(A, A_1, A_2)$ to be a point of local minimum or a point of local maximum.}

\medskip

We show that, if the admissible triple
$(A, A_1,A_2)$ is a point of local
extremum  (maximum or minimum)
and the curvature of $\g$ at $A$ is not equal to zero, then the following three
lines either have a common point or are parallel: the normal $n_{A_1}$ to
$\g_1$ at $A_1$, the normal $n_{A_2}$ to $\g_2$ at $A_2$, and the line
$n_{C_{A}}$ perpendicular to $A_1A_2$ and passing through the center of curvature $C_{A}$ of $\g$ at $A$. We call this necessary condition for optimality the {\it Three Concurrent Normals Property} and refer to it, in this Introduction, as the 3-CNP.

\medskip

If $\g$ is a circle
with center  $O$, then the normals to $\g$ contain the point $O$, and the considered optimization problem reduces to the following one:

\medskip

{\it Find necessary conditions for a non-degenerate segment
$[A_1A_2]$ with $A_1\in\g_1$, $A_2\in\g_2$,
and $O \in A_1A_2$, to be locally shortest (or locally longest)}

\medskip

\noindent Since $O$ is the center of curvature for all points of the circle $\g$, the 3-CNP now takes the form: for every locally optimal  (maximal or minimal) segment
$[A_2A_2]$ the normal $n_{A_1}$ to
$\g_1$ at $A_1$, the normal $n_{A_2}$ to $\g_2$ at $A_2$, and the line $n_O$ through $O$
perpendicular to $A_1A_2$ either intersect at one point or are parallel.

The particular case of this latter problem, when the curves
$\g_1$ and $\g_2$ are the arms of an angle, $O$ is an interior point for this angle, and one looks for the shortest segment
$[A_1A_2]$ cut from the angle by a line $A_1A_2$ containing $O$, is trully classical.
Neuberg  \cite{NBRG} (1907) mentions that this problem can be found in  Newton's
{\it Opuscules, t.I, p.87}.
It has also been considered by many other authors.
Among the earliest of them we mention here  L'Huilier \cite{SL,SL1}(1795, 1811), Neovius \cite{EN}(1887) and  J. Casey \cite{JCas} (1886).
In this simple minimization problem, the 3-CNP is not only a necessary but also a sufficient condition for the optimality of the line $A_1A_2$. Another, more widely known, characterization of the optimal line $A_1A_2$ for this ``angle problem'' is the following {\it isotomic property}: the foot of the perpendicular from the vertex of the angle to $A_1A_2$ and the point $O$ are symmetric to each other with respect to the middle of the segment $[A_1A_2]$. A line $A_1A_2$ with this isotomic property is frequently called ``Philon line'' (or ``Philo's line'') because the ancient
engineer Philon, who lived in the 3rd century BC, used a line with this property to provide one more solution to the problem of cube doubling  (see \cite{Wiki1} and \cite{Wiki2}). As shown by Coxeter and van de Craats in \cite{CC}, this
isotomic property is also valid for angles (and points) in the spherical and hyperbolic planes.

\medskip

Another particular case of the above problem is when both
$\g_1$ and $\g_2$ coincide with the boundary of a closed and bounded convex set in the plane, and $O$ is a point from its interior.
If we are looking for a locally shortest segment
$[A_1A_2]$ cut from the convex set by a line passing
through $O$, then 3-CNP remains valid  (with $n_{C_{A}}$ replaced by $n_O$) for every optimal line, even without
any smoothness assumptions imposed on the boundary of the convex set. It should be noted, however, that, in this case, the local minimality of the segment
$[A_1A_2]$ implies smoothness of the boundary of the convex set at
the endpoints $A_1$ and $A_2$.  In particular, if the convex set is a polygon, then the endpoints of any locally minimal segment
 $[A_1A_2]$ are not among the vertices of the polygon.

For the related maximization problem in this ``point-convex'' setting, one does need the smoothness of the boundary of the convex set to prove that 3-CNP is valid for any locally optimal line. For more information about this ``point-convex'' particular  case of the problem considered in this article, as well as for its natural multidimensional generalizations,  the reader is referred to \cite{KMN}(2027).

\medskip

Returning to the ''three curves situation'' with
$\g_1, \g_2$, and $ \g$, one may consider all lines $A_1A_2$, where $A_1 \in \g_1$, $A_2 \in \g_2$, which are tangent to $\g$, and look for the shortest segment $[A_1A_2]$.
At the end of his article \cite{EVES}(1959), Howard Eves mentions this problem as a possible generalization of the angle problem considered above and notes that for any optimal line $A_1A_2$, the three normals  $n_{A_1}$ to
$\g_1$ at $A_1$, $n_{A_2}$ to $\g_2$ at $A_2$, and the normal
$n_{C}$ to $\g$ at  the point $C$ where the line $A_1A_2$ touches $\g$ are concurrent. He adds ``This generalization seems to be due to Isaak Newton'' without giving further information under what smoothness requirements this had been proved and where it had been published. We show in this article that this ``tangential'' minimization problem, as well as its maximization counterpart, is closely related, and even equivalent (under some additional smoothness assumptions), to the optimization problem formulated at the beginning of the paper.

 \medskip

 A significant part of what follows below is devoted to the detailed consideration of the
 ``degenerate cases'' naturally appearing in this setting
(when, for instance, $\g_1$ or $\g_2$ coincides with $\g$ or when the optimal line $A_1A_2$ is tangent to one or both $\g_1$ and $\g_2$).

\medskip

To prove the above-mentioned results, we reparametrize the curves $\g_1$ and $\g_2$ by the
natural arc-length parameter $s$ of the curve $\g$
in such a coordinated way that, for any $s$,  the line $A_1(s)A_2(s)$, where $A_i(s) \in \g_i, i=1,2,$ is the point corresponding to the parameter $s$ in the newly parametrized curve $\g_i$, is the normal to $\g$ at $A(s) \in \g$. Then, we use the standard necessary condition for optimality: {\it the derivative with respect to $s$ of the corresponding distance function is zero}.
The assumptions under which the desired reparametrization is possible are presented in Section 2. Section 3 contains the proofs of the
main results. The role of Section 4 (the short Appendix) is to demonstrate that, when the curvature of
$\g$ is not vanishing, the angle $\theta$ a line tangent to $\g$ concludes with the abscissa axis (or any other fixed line) is also a natural parameter for the curve $\g$. The derivatives of the involved distances with respect to $\theta$ sometimes deliver simpler formulas that make the geometric basis for the validity of 3-CNP and some other geometric properties more transparent.

\section{Notations and auxiliary results}\label{au}

In what follows by a $C^k$-smooth curve ($k\in\Bbb N \cup
\{{\infty}\}$) we will mean a $C^k$-smooth mapping
$\g:\Delta\to\Bbb R^2$
where $\Delta \subset
\mathbb{R}$ is an open interval and $\g\,'(t) \neq 0$ for every $t \in \Delta$
(i.e. the curves we consider are ``regular'').
According to a common practice, we sometimes identify $\g$ with its image $\{\g(t): t \in \Delta \}$ in $\Bbb R^2$.

Each curve $\g(t)$, $t \in \Delta$, can be represented in the form
$\g(t)=(p(t),q(t))$ where $p(t)$ and $q(t)$ are the coordinates of the point $\g(t)$. Then the {\it tangent to $\g$} at the point $\g(t_0)=(p(t_0), q(t_0)), t_0 \in \Delta$, is the line $t_{\g(t_0)}$ through
$\g(t_0)$ which is parallel to the nonzero {\it  tangent vector}
$T(t_0) := (p\,'(t_0), q\,'(t_0))= \g\,'(t_0)$ at the same point,  i. e.
$$t_{\g(t_0)}=\{\g(t_0) + \lambda T(t_0) : \lambda \in \Bbb R\}.$$

The {\it normal to $\g$} at $\g(t_0)$  is the line
$n_{\g(t_0)}$ through  $\g(t_0)$ that is parallel to the {\it normal vector} $N(t_0):=(- q\,'(t_0), p\,'(t_0))$ at  the point $\g(t_0)$, i. e.
$$n_{\g(t_0)}=\{\g(t_0) + \lambda N(t_0) : \lambda \in \Bbb R\}.$$
Clearly, $n_{\g(t_0)}$ is perpendicular to $t_{\g(t_0)}$ and  consists of the points $(x,y) \in \Bbb R^2$ such that

\begin{equation}\label{normala}
(x - p(t_0))p\,'(t_0) + (y - q(t_0))q\,'(t_0) = 0.
\end{equation}

Recall that the signed curvature $\kappa_{\g}(t)$ at some point $\g(t)=(p(t),q(t))$, $t \in \Delta$, of a $C^2$-smooth
plane curve  $\g$  is given by

\begin{equation}\label{1}
\kappa_{\g}(t)=\frac{p\,'(t)q\,''(t)-p\,''(t)q\,'(t)}{(p\,'(t)^2+q\,'(t)^2)^{3/2}} .
\end{equation}


If $s$ is the arc-length parameter of $\g$, then
\begin{equation}\label{2}
\kappa_{\g}(s)=p\,'(s)q\,''(s)-q\,'(s)p\,''(s)
\end{equation}
since $p\,'(s)^2+q\,'(s)^2=1$. If
$\kappa_{\g}(s_0)\neq 0$ then the center of
curvature $C_{\g(s_0)}$ of $\g$ at $\g(s_0)$ is the point
\begin{equation}\label{3}
C_{\g(s_0)}=(p(s_0)-\frac{q\,'(s_0)}{\kappa_{\g}(s_0)} \ , \
q(s_0)+\frac{p\,'(s_0)}{\kappa_{\g}(s_0)}).
\end{equation}
The center of curvature $C_{\g(s_0)}$ lies on the normal $n_{\g(s_0)}$ to $\g$ at
$\g(s_0)$ and the distance between $\g(s_0)$
and $C_{\g(s_0)}$ is equal to $|\frac{1}
{\kappa_{\g }(s_0)}|$ because $N(s_0)$ is a vector of lenght one.

Let $\g_1(t)$, $t \in \Delta_1$, and $\g(s)$, $s \in \Delta$, be $C^1$-smooth curves
with  $s$ being the arc-length parameter of  $\g$.  Consider the function
\begin{equation}\label{G}
G(t,s)=(p_1(t)-p(s))p\,'(s)+(q_1(t)-q(s))q\,'(s)
\end{equation}
defined in $\Delta_1 \times \Delta$.

The next simple statement summarizes some of the relations between the function $G$ and its partial derivatives
 $G'_t$ and $G'_s$ on the one hand, and the geometric essence of the problems considered in this paper on the other.

\begin{lemma}\label{l1}
(i) A pair $(t_0, s_0)$ is a solution to the equation $G(t,s)=0$ if and only if $\g_1(t_0) \in n_{\g(s_0)}$.

(ii) A pair $(t_0, s_0)$ is a solution to the system of two equations $G(t,s)=0$,
$G'_t(t,s)=0$ if and only if
$n_{\g(s_0)} = t_{\g_1(t_0)}$.

(iii)  If $\g(s)$, $s \in \Delta$, is a $C^2$-smooth curve then a pair $(t_0, s_0)$ is a solution to the system of two equations $G(t,s)=0$,
$G'_s(t,s)=0$ if and only if $\kappa_{\g}(s_0) \neq 0$ and $\g_1(t_0) = C_{\g(s_0)}$.

\end{lemma}

\begin{proof}
Item (i) follows trivially from
the definition of the normal $n_{\g(s_0)}$ in terms of equation (\ref{normala}).

The relation
$ G'_t(t_0,s_0)= p\,'_1(t_0)p\,'(s_0)+q\,'_1(t_0)q\,'(s_0)=0 $ shows that the non-zero tangent vectors
$\g\,'_1(t_0)=(p\,'_1 (t_0), q\,'_1(t_0))$ and
$\g\,'(s_0)=(p\,'(s_0), q\,'(s_0))$ are orthogonal to each other. Combined with the fact that $\g_1(t_0) \in n_{\g(s_0)}$, this yields (ii).

To prove (iii), we note that
\begin{equation}\label{4}
G'_s(t_0,s_0)=
(p_1(t_0)-p(s_0))p\,''(s_0)+(q_1(t_0)-q(s_0))q\,''(s_0)-1
\end{equation}
 because $p\,'(s_0)^2+ q\,'(s_0)^2=1$.
 If $G'_s(t_0,s_0)=0$, equation (\ref{4}) implies that the vector
 $\g\,''(s_0)=(p\,''(s_0), q\,''(s_0)) \neq (0,0)$.
 Hence,
 $\kappa_{\g}(s_0) \neq 0$ and the point $C_{\g(s_0)}$ exists.

Taking into account the Frenet-Serret formula
\begin{equation}\label{5}
(p\,''(s),q\,''(s)) = \kappa_{\g}(s)(-q\,'(s), p\,'(s))
\end{equation}
we get
\begin{equation}\label{6}
G'_s(t_0,s_0)=
\kappa_{\g}(s_0)[(p_1(t_0)-p(s_0))(-q\,'(s_0))+(q_1(t_0)-q(s_0))p\,'(s_0)]-1
\end{equation}
Having in mind (\ref{3}), i.e. that $C_{\g(s_0)} - \g(s_0) = \frac{1}{\kappa_{\g}(s_0)}N(s_0)$,
this equation can be rearranged in terms of the scalar product of vectors as

\begin{equation}\label{7}
G'_s(t_0,s_0)=\kappa_{\g}(s_0) \langle(\g_1(t_0) - \g(s_0)), N(s_0)\rangle - 1= \kappa_{\g}(s_0) \langle(\g_1(t_0) - C_{\g(s_0)}), N(s_0)\rangle.
\end{equation}
This suffices to complete the proof of item (iii) because the vectors $N(s_0)$, $\overrightarrow{C_{\g(s_0)}\g_1(t_0)}$, and
$\overrightarrow{\g(s_0)\g_1(t_0)}$ are colinear if and only if $G(t_0,s_0)=0$ (follows from item (i) ).
\end{proof}

Next, we provide sufficient conditions under which a curve
$\g_1(t), t \in \Delta_1$, can be locally reparametrized by the  arc-length parameter $s$ of another curve $\g(s), s \in \Delta$, and,
vice-versa, the curve $\g(s)$ can be locally reparametrized by the parameter $t$ of the curve $\g_1(t)$.

\begin{prop} \label{nnorm1}
 Let $\g_1(t)$, $t \in \Delta_1$, and $\g(s)$,
 $s \in \Delta$, be  $C^1$-smooth curves where  $s$ is the arc-length parameter of
 $\g$. Let $A_1= \g_1(t_0)$ and $A= \g(s_0)$ be such that $A_1 \in n_A$.

(i) If $n_A$ is not tangent to $\g_1$ at $A_1$,
then there exist an open interval $U \subset \Delta$ containing $s_0$, and a unique $C^1$-smooth function $t_1(s)$ defined in $ U$ with values in
$\Delta_1$ such that $t_1(s_0) = t_0$ and for every
$s \in U$ the point $\g_1(t_1(s))$ belongs to the normal $n_{\g(s)}$.

(ii) If $\g(s)$, $s \in \Delta$, is a $C^2$-smooth curve and either $\kappa_{\g}(A)=0$ or $A_1\neq C_A$, then there exist an open interval $V \subset \Delta_1$ containing $t_0$, and a unique $C^1$-smooth function $s_1(t)$ defined in $V$ with values in
$\Delta$ such that $s_1(t_0) = s_0$ and for every
$t \in V$ the point $\g_1(t)$ belongs to the normal $n_{\g(s_1(t))}$.
\end{prop}

\begin{proof}
(i) As $A_1 \in n_A$ we get from item (i) of Lemma \ref{l1} that
$G(t_0,s_0)=0$. Since $n_A$ is not tangent to $\g_1$ at $A_1$ we get from item (ii) of Lemma \ref{l1} that $G'_t(t_0,s_0) \neq 0$.
The implicit function theorem provides us with an open interval $U$, $s_0 \in U \subset \Delta$, and with a unique continuously differentiable function $t_1(s)$ defined in $U$ with values in
$\Delta_1$ such that $t_1(s_0)=t_0$ and $G(t_1(s), s) =0$ for every $s \in U$. This means,  $\g_1(t_1(s))$ belongs to the normal $n_{\g(s)}$ for every $s \in U$. The proof of item (i) is completed.

The proof of (ii) is similar. If $\kappa_{\g}(A)=0$ or $A_1\neq C_A$, we get from the proof of item (iii), Lemma \ref{l1},  that in both cases $G'_s(t_0,s_0) \neq 0$. Using the implicit function theorem again, we get an open interval $V $, $t_0 \in V \subset \Delta_1$, and a unique continuously differentiable function $s_1(t)$ defined in $V$ with values in $\Delta$ such that $s_1(t_0)=s_0$ and
$G(t,s_1(t)) =0$ for every $t\in V$.
This means, $\g_1(t)$ belongs to the normal $n_{\g(s_1(t))}$ for every $t \in V$. This completes the proof of (ii).

\end{proof}

\begin{prop}\label{nnorm2}
 Let $\g_1(t)$, $t \in \Delta_1$, and $\g(s)$,
 $s \in \Delta$, be  $C^1$-smooth curves (with  $s$ being the arc-length parameter of
 $\g$). Let the points $A_1= \g_1(t_0)$ and $A= \g(s_0)$ be such that $A_1 \in n_A$.

If $n_A$ is not tangent to $\g_1$  at $A_1$, then  the following properties are equivalent:

(a) The derivative $t_1'(s_0) \neq 0$ where
$t_1: U \rightarrow \Delta_1$ is the $C^1$-smooth function from item (i) of Proposition
\ref{nnorm1};

(b) There exists an open subinterval $U'$, $s_0 \in U'\subset U$, on which $t_1$ is a difeomorphism (i.e. the restriction of $t_1$ on $U'$ is invertible and its inverse $t_1^{-1}: t(U') \rightarrow U' $  is a  $C^1$-smooth mapping).

If $\g$ is a $C^2$-smooth curve, then each of the items (a) and (b) is equivalent to the next property

(c)  Either $C_A$ does not exist (i.e.
$\kappa_{\g}(A)=0$) or $A_1\neq C_A$.

\end{prop}
\begin{proof} If $t_1'(s_0) \neq 0$  then $t_1'(s) \neq 0$ for all $s$ from some sub-interval $U'$, $s_0 \in U'\subset U$. This means, the mapping $t_1:U' \to t_1(U')$ is strictly monotone in $U'$ and, therefore, invertible. Moreover, since $t_1$ is continuously differentiable with non-vanishing derivative, its inverse
$t_1^{-1}:t_1(U') \rightarrow U'$ is continuously differentiable as well, and has a non-vanishing derivative, i.e., (a) and (b) are equivalent.

We show now that (a) and (c) are equivalent if $\g$ is a $C^2$-smooth curve.
By assumption,
$n_A$ is not tangent to $\g_1$ at $A_1$. Hence, by item (ii) of Lemma \ref{l1} we have $G'_t(t_0,s_0) \neq 0$. Consider the function $t_1(s)$ from item (i) of Proposition \ref{nnorm1} and differentiate the identity  $G(t_1(s), s) =0$. We get, for $s=s_0$, the equation

\begin{equation}\label{9}
G'_t(t_1(s_0),s_0)t_1'(s_0) + G'_s(t_1(s_0),s_0)=G'_t(t_0,s_0)t_1'(s_0) + G'_s(t_0,s_0) = 0.
\end{equation}

It follows that $t'_1(s_0) \neq 0$ if and only if $G_s'(t_0,s_0) \neq 0$, which, by virtue of item (iii) of Lemma \ref{l1}, is equivalent to  (c).
\end{proof}
The next statement is a ``mirror version'' of the previous one, and its proof is omitted.

\begin{prop} \label{nnorm3}
Let $\g_1$ be a $C^1$-smooth curve, $\g$ be a $C^2$-smooth curve, $A =\g(s_0)$, and
$A_1=\g_1(t_0)$.
If either $\kappa_{\g}(s_0)=0$ or $A_1\neq C_A$ then  the following properties are equivalent:

(a) The derivative $s_1'(t_0) \neq 0$ where  $s_1: V \rightarrow \Delta$ is the $C^1$-smooth function from item (ii) of Proposition
\ref{nnorm1};

(b) there ixists an open subinterval $V'$, $s_0 \in V'\subset V$, on which $s_1$ is a difeomorphism (i.e. the restriction of $s_1$ on $V'$ is invertible and its inverse $s_1^{-1}: s(V') \rightarrow V' $  is a  $C^1$-smooth mapping);

(c) $n_A$ is not  tangent to $\g_1$  at $A_1$

\end{prop}

\begin{cor} \label{homo}
 Let $\g_1(t)$, $t \in \Delta_1$ be a $C^1$-smooth curve, $\g(s)$,  $s \in \Delta$, be  $C^2$-smooth curve and let $A_1= \g_1(t_0)$ and $A= \g(s_0)$ be such that $A_1 \in n_A$.

If either $\kappa_{\g}(A) = 0$ or $A_1\neq C_A$, and $n_A$ is not tangent to $\g_1$ at $A_1$,  then there exists an interval $U''$,  $s_0 \in U'' \subset \Delta$, such that the restriction of $t_1$ on $U''$ and the restriction of $s$ on $t_1(U'')$ are inverse to each other.
\end{cor}

\begin{proof}
Consider the inverse mapping $t_1^{-1}:t(U') \rightarrow U'$ from Proposition \ref{nnorm2}. Its restriction on $V' \cap t(U')$, where $V'$ is from Proposition \ref{nnorm3}, is a second implicit function (the first being the restriction of $s_1$ on this set). Uniqueness of the implicit function implies that $s_1$ and $t_1^{-1}$ coincide on the set $V' \cap t(U')$.

\end{proof}

We will also need the following simple observation.

\begin{lemma}\label{lem} Let $\g_1(t)$, $t \in \Delta_1$, and $\g_2(\tau)$, $ \tau \in \Delta_2$, be $C^1$-smooth curves. Let $t_0 \in W$ where
$W$ is an open subset of $\Delta_1$ and let
$u:W \to \Delta_2$ be a $C^1$-smooth function such that $\g_1(t_0) \neq \g_2(u(t_0))$ and $u\,'(t_0)=0$.
Set $d(t)= |\g_1(t)\g_2(u(t))|$ for any $t \in W$.
Then $d'(t_0)=0$ if and only if the line
$\g_1(t_0)\g_2(u(t_0))$ is the normal to $\g_1$ at $\g_1(t_0)$.
\end{lemma}

\begin{proof} Since $u'(t_0) =0$, we get for the derivative  at $t_0$ of the function
$$d(t)^2 = \langle \g_1(t) - \g_2(u(t)), \g_1(t) - \g_2(u(t))\rangle$$
the expression
$$ d(t_0) d'(t_0) = \langle \g_1(t_0) - \g_2(u(t_0)), \g_1'(t_0) \rangle.$$
As $d(t_0) \neq 0$ we see that
 $d'(t_0)=0$ if and only if the non-zero vector
 $\g_1(t_0) - \g_2(u(t_0))$ is perpendicular to the non-zero tangent vector $\g\,'_1(t_0)$.
\end{proof}

\section{Main results}\label{mr}

Let $\g_1$ and $\g_2$ be $C^1$-smooth curves in the plane.
In this section, we obtain geometric information about the points of local extremum of the distance function
$d(A_1,A_2)=|A_1A_2| , A_1\in \g_1 , A_2\in \g_2$, $A_1 \neq A_2$, under the following constraints:

(i) The lines $A_1A_2$ are normal to a given $C^2$-smooth curve $\g$.

(ii) The lines $A_1A_2$ pass through a given point $O$.

(iii) The lines $A_1A_2$ are tangent to a given $C^1$-smooth curve $\g$.

(iv) The lines $A_1A_2$ intersect a given $C^1$-smooth curve $\g$.

To distinguish these four cases, we denote the corresponding
distance functions by $d_{n_{\g}}, d_O \ , d_{t_{\g}}, $ and
$d_{\g}$.

We first consider the function $d_{n_{\g}}$.

\begin{thm}\label{A} Let $\g_1(t)$,
$t\in \Delta_1$, and $ \g_2(\tau)$,
$\tau \in \Delta_2$,  be $C^1$-smooth curves
and $\g(s)$, $s \in \Delta$, be a $C^2$-smooth curve with sign curvature
$\kappa_{\g}$. Let $A_1\in \g_1$ and $A_2\in \g_2$ be such that
$A_1\neq A_2$, $A_1A_2=n_A$ for some $A \in \g$, and the function $d_{n_{\g}}$ has a local extremum at the triple $(A, A_1,A_2)$.

(i) If $\kappa_{\g}(A)\neq 0$, then the normals $n_{A_1} ,
n_{A_2}$ and the line $n_{C_A}$ through the center of curvature
$C_A$ of $\g$ and perpendicular to $A_1A_2$ either intersect at
one point, or are parallel.

(ii) If $\kappa_{\g}(A)= 0$, then the normals $n_{A_1}$ and
$n_{A_2}$ are parallel.
\end{thm}

\begin{proof}To prove (i), we consider the following three options for the line $A_1A_2$ that cover all possible cases:

\emph{1. $A_1A_2$ is tangent to $\g_1$ and  $\g_2$
at $A_1$ and  $A_2$ correspondingly.}

\emph{2. $A_1A_2$ is not tangent to $\g_1$ and
$\g_2$ at $A_1$ and $A_2$ correspondingly. }

\emph{3. $A_1A_2$ is tangent to $\g_1$ at $A_1$ and not tangent to $\g_2$ at $A_2$. }

In the case \emph{1} the normals $n_{A_1}$ and $n_{A_2}$ are parallel to
the line $n_{C_A}$ and there is nothing to prove.

In case \emph{2}, we will prove that the normals
$n_{A_1} , n_{A_2}$
and the line $n_{C_A}$ intersect at one point.
We split case \emph{2} into two separate sub-cases:

(a) $C_A$ is different from the points $A_1$ and $A_2 $ (this is the so-called ``generic case'');

(b) $C_A$ coincides with one of the two different points $A_1$ and $A_2$.

Consider first the sub-case (a).
Let $s$ be the arc-length parameter of $\g$ with $A=\g(0)$
and set $\g(s)=(p(s), q(s)) \ , \ T(s)=\g\,'(s)=(p\,'(s),q\,'(s))$ and
$N(s)= (-q\,'(s), p\,'(s)$. As already mentioned above,  $T(s)$ is the unit tangent vector
and $N(s)$ is the unit normal vector to $\g$ at $A(s)=\g(s)$. Since the line $A_1A_2$ is not tangent to $\g_1$ at $A_1$, it
follows by item (i) of Proposition \ref{nnorm1}, that there is  a
$C^1$-smooth function $t_1(s)$, defined in a neighborhood $U$ of $0$, such that $A_1= \g_1(t_1(0))$ and
$\g_1(t_1(s))= \g(s)+ d_1(s)N(s)$ for every $s \in U$.
Here
$d_1(s):= \langle \g_1(t_1(s)) - \g(s), N(s)\rangle$ is
the {\it oriented distance} from $A(s)=\g(s)$ to $A_1(s)=\g_1(t_1(s))$ relative to the orientation
of the line $A_1A_2$ given by the vector $N(s)$. Set
$(T,N)=(T(0),N(0))$ and let $B_1$ be the intersection point of the
normal $n_{A_1}$ and the line $n_{C_A}$. This intersection point $B_1$ exists because $n_A$ is not tangent to $\g_1$ at $A_1$. We show next that
\begin{equation}\label{D1}
\overrightarrow{C_AB_1} =-\frac{d\,'_1(0)}{\kappa_{\g}(A)}T .
\end{equation}

To prove this, we differentiate the identity
$\g_1(t_1(s))= \g(s)+ d_1(s)N(s)$ and use the Frenet-Serret formula $N'(s)= -k(s)T(s)$ for plane
curves. This yields,  for the derivative at $s=0$, the expression

$$\frac{d}{ds}\g_1(t_1(s))|_{s=0}=
T+d\,'_1(0)N+d_1(0)N'(0)=T(1-d_1(0)\kappa_{\g}(A))+d\,'_1(0)N.$$

Set $\overrightarrow{C_AB_1} =\lambda T$. Then
$$ \overrightarrow{A_1B_1}= \overrightarrow{A_1C_A} +
\overrightarrow{C_AB_1}=
(\frac{1}{\kappa_{\g}(A)}-d_1(0))N+\lambda T$$ and, having in mind
that $C_A\neq A_1$ (i.e.
$\frac{1}{\kappa_{\g}(A)} \neq d_1(0)$), we conclude that $A_1B_1$ is perpendicular to
$\frac{d}{ds}\g_1(t_1(s))|_{s=0}$
if and only if $\lambda = -\frac{d'_1(0)}{\kappa(A)}$.

Let now $B_2$ be the intersection point of the normal $n_{A_2}$
and the line $n_{C_A}$. Then the same reasoning as above shows
that
\begin{equation}\label{D2}
\overrightarrow{C_AB_2} =-\frac{d\,'_2(0)}{\kappa_{\g}(A)}T ,
\end{equation}
where $d_2(s)$ is the oriented distance from $A(s)=\g(s)$ to
$A_2(s)=\g_2(\tau_2(s))$. The function
$\tau_2(s)$ here is obtained by item (i) of Proposition \ref{nnorm1} with the curve $\g_1$ replaced by the curve $\g_2$. It follows by (\ref{D1}) and (\ref{D2})
that
\begin{equation}\label{D3}
\overrightarrow{B_1B_2} =\frac{d\,'_1(0)-d\,'_2(0)}{\kappa_{\g}(A)}T.
\end{equation}
Let $d(s)$ be the oriented distance from $A_1(s)=\g_1(s)$ to
$A_2(s)=\g_2(s)$. Then $d(s)=d_2(s)-d_1(s)$ and
$d\,'_2(0)-d\,'_1(0)=d\,'(0)=0$ since $d(s)$ has a local extremum at
$s=0$. So, $B_1=B_2$ as desired.

In the sub-case (b) of case \emph{2} we can assume, without loss of generality, that $C_A= A_2$. This point belongs to
$n_{A_2}$ and to $n_{C_A}$. It remains to show that the line $A_1A_2$ is the normal to $\g_1$ at the point $A_1$. This will be derived from Lemma \ref{lem}. Indeed, as $A_1 \neq C_A$, it follows by item (ii) of Proposition \ref{nnorm1} and Corollary \ref{homo} that there exist an open interval
$V \subset \Delta_1$ containing $t_0$, and a  difeomorphism $s_1(t)$ defined in $V$ with values in
$\Delta$ such that $s_1(t_0) = 0$,
$A_1 =\g_1(t_0)$,
$A= \g(0)$, and for every
$t \in V$ the point $\g_1(t)$ belongs to the normal
 $n_{\g(s_1(t))}$. Note that the image $s_1(V)$ of the set $V$ is an open neighborhood of $0 $ in $\Delta$.

 Since the line $A_1A_2$ is not tangent to $\g_2$ at $A_2$, item (i) of Proposition \ref{nnorm1} implies that there exist an open interval $U \subset \Delta$ containing $0$, and a $C^1$-smooth function $\tau_2(s)$ defined in $ U$ with values in
$\Delta_2$ such that $\tau_2(0) = \tau_0$, $\g_2(\tau_0)=A_2$, and for every
$s \in U$ the point $\g_2(\tau_2(s))$ belongs to the normal $n_{\g(s)}$. Note that, by Proposition \ref{nnorm2} the derivative $\tau_2'(0) = 0$. We can assume, without restricting the generality, that
$s_1(V) \subset U$. This allows to consider the composed function $u(t) = \tau_2(s_1(t))$. Clearly, $u'(t_0) = 0$. Since $A_1A_2$ is a local extremum,
$d'(t)=0$ where $d(t)=|\g_1(t)\g_2(u(t))|$. Lemma \ref{lem} implies that the line $\g_1(t)\g_2(u(t)$ is the normal to $\g_1$ at the point $A_1$.

Let us now consider the case \emph{3} of item (i) from Theorem \ref{A}: the line $A_1A_2$ is tangent to $\g_1$ at $A_1$ and not tangent to
$\g_2$ at $A_2$.

If $A_1=C_A$, then
$n_{A_1}=n_{C_A}$ and  $n_{A_1}, n_{A_2}$ and $n_{C_A}$ intersect at one point because $A_1A_2$ is not tangent to $\g_2$ at $A_2$.

Let now $A_1\neq C_A$. Consider the function
$s_1(t)$ defined in item (ii) of Proposition \ref{nnorm1}. Note that (by Proposition
\ref{nnorm3}) $s_1'(t_0)=0$. Consider also the function $\tau_2(s)$ the existence of which is assured by item (i) of Proposition \ref{nnorm1}
(because $A_1A_2$ is not tangent to $\g_2$ at $A_2$.). Put $u(t):=\tau_2(s_1(t))$. Clearly, $u'(t_0)=0$. Lemma \ref{lem} implies that $A_1A_2$ is the normal to $\g_1$ at $A_1$. This is a contaradiction because $A_1A_2$ cannot be both a tangent and a normal to $\g_1$ at $A_1$ (because $\g_1'(t_0) \neq 0 $).

\medskip
It remains to prove the claim in item (ii) of  Theorem \ref{A} (when
$\kappa_{\g}(A)=0$).  There are, as above, three options for the line $A_1A_2$ which cover all possible cases:

\emph{1'. $A_1A_2$ is tangent to $\g_1$ at $A_1$ and  tangent to $\g_2$ at $A_2$.}

\emph{2'. $A_1A_2$ is not tangent to $\g_1$ at $A_1$ and not tangent to $\g_2$
at $A_2$. }

\emph{3'.}$ A_1A_2$ is tangent to $\g_1$ at $A_1$ and not tangent to $\g_2$
at $A_2$.

We proceed as in the proof of item (i).  The claim in the case \emph{1' } needs no proof. In \emph{3'} we reach a contradiction as in the case  \emph{3} of item (i). Therefore, it suffices to consider only the case \emph{2'} when the line
$A_1A_2$ is not tangent to $\g_1$ at $A_1$
and not tangent to $\g_2$ at $A_2$. We have seen that there exists an open set
$U$, $0 \in U \subset \Delta$, and two functions
$t_1(s)$ and $\tau_2(s)$ in $U$ with values in
$\Delta_1$ and $\Delta_2$ correspondingly,  such that $A_1=\g_1(t_1(0))$.
$A_2 = \g_2(\tau_2(0))$
and, for every $s \in U$, the points
$A_1(s):=\g_1(t_1(s)) $ and $A_2(s):=\g_2(\tau_2(s))$ belong to $n_{A(s)}$
where $A(s) = \g(s)$. Clearly,
$\g_1(t_1(s)) = \g(s) +d_1(s) N(s)$, where $d_1(s)$ is the oriented distance between
$A(s) = \g(s)$ and $A_1(s)=\g_1(t_1(s))$.  As
$N'(0)=-\kappa_{\g}(A)T=0$ we get for the derivative of $\g_1(t_1(s)) $ at $s=0$ the expression
$$\frac{d}{ds}\g_1(t_1(s))|_{s=0}=
T+d\,'_1(0)N.$$

Similarly, we have

$$\frac{d}{ds}\g_2(\tau_2(s))|_{s=0}=
T+d\,'_2(0)N$$
where $d_2(s)$ is the oriented distance between $A(s)$ and $A_2(s)$.
Since the triple $(A, A_1,A_2)$ is a local extremum
we have, as in item (i), that $d_1'(0)=d_2'(0)$. This implies that the tangent vector to
$\g_1(t)$ at $A_1$ and the tangent vector to $\g_2(\tau)$ at $A_2$ are colinear. It follows that the normals to the curves at $A_1$ and $A_2$ are parallel.
\end{proof}

The following example shows that the most we can say in case (ii)
of Theorem \ref{A} when $\kappa_{\g}(A)=0$ is that the normals
$n_{A_1}$ and $n_{A_2}$ are parallel.
\smallskip

\begin{exam}\label{ex1} Let $\g_1$ and $\g_2$ be parallel lines and $\g$ be a line
which is neither parallel nor perpendicular to $\g_1$ and $\g_2$.
Take an arbitrary point $A\in \g$ and let $A_1$, and $A_2$ be the
intersection points of the normal $n_A$ to $\g$ at $A$ with $\g_1$
and $\g_2$, respectively. Then the admissible triple $(A, A_1, A_2)$ is a point of global extremum for
$d_{n_{\g}}$ since this function is constant. On the
other hand, $\kappa_{\g}(A)=0$ and $C_A$ does not exist, but the normals
$n_{A_1}, n_{A_2}$ and any line perpendicular to $A_1A_2$ are
neither parallel nor intersect at one point.
\end{exam}
\smallskip

If $\g_2=\g$ the admissible triples are of the type
$(A, A_1, A)$ and we work with pairs
$(A_1,A)$, such that $A_1\in \g_1 \ , \ A\in \g $ and $A_1A$ is the normal $n_A$ to $\g$ at $A$. In this case, the distance $d_{n_{\g}}(A_1,A) =|A_1A|$ and we have the following corollary of Theorem \ref{A}.

\begin{cor}\label{cor1} Let $\g_1$ be a $C^1$-smooth curve and
$\g$ be a $C^2$-smooth curve. Consider all pairs
$(A_1,A)$ such that $A_1\in \g_1$, $A\in \g$,
$A \neq A_1\neq C_A$, and $A_1A=n_A$. If the function
$d_{n_{\g}}$ has a local extremum at such a pair
$(A_1, A)$, then
$n_{A_1}=n_A$. I.e. $A_1A$ is a common normal to $\g_1$ and $\g$
at $A_1$ and $A$, respectively.
\end{cor}

\begin{proof} To apply Theorem \ref{A} put
$\g_2=\g$ and $A_2=A$. Then $n_{A_2}=n_A$. If $\kappa_{\g}(A)\neq 0$,
then $n_{A_2}\cap n_{C_A}=C_A$ and it follows from Theorem \ref{A} (i) that $n_{A_1}$ contains $C_A$. It follows that $n_{A_1}=n_A$.
If $\kappa_{\g}(A)=0$, then by Theorem \ref{A} (ii) we have that $n_{A_1}$ is parallel to $n_{A_2}=n_A$
and again $n_{A_1}=n_A$ since $A_1\in n_A$.
\end{proof}

The next example shows that the condition $A_1\neq C_A$ is
essential for the validity of Corollary \ref{cor1}.
\smallskip

\begin{exam}\label{ex2} Let $\g= \g_2$ be a unit circle with center at some point
$A_1$, $A$ be a point from $ \g=\g_2$ and $\g_1$ be a $C^1$-smooth curve through $A_1$ such that
$A_1A$ is neither tangent nor normal to $\g_1$
at $A_1$. Then the pairs $(A_1', A')$ of interest to us are of the form
$(A_1, A')$ and the distance
$d_{n_{\g}}$ is constant (equal to 1).
The three normals from Theorem \ref{A} meet at $A_1$ but the line $AA_1$ is not a common normal for $\g$ and $\g_1$ as in Corollary \ref{cor1}.

\end{exam}
\smallskip

Now we consider some consequences of the above results for the points of
local extremums of the functions $d_O$, $d_{\g}$ and $d_{t_{\g}}$.
\medskip
\begin{cor}\label{cor2}  Let $\g_1$ and $ \g_2$ be $C^1$-smooth
curves and $O$ be a fixed point in the plane. Consider all pairs $(A_1,A_2)$ such that
$A_1\in \g_1$, $A_2\in \g_2$, $A_1\neq A_2$ and $O\in A_1A_2$. If
the function $d_O$ has a local extremum at $(A_1,A_2)$, then the
normals $n_{A_1} , n_{A_2}$ and the line $n_O$ through $O$ that is
perpendicular to $A_1A_2$ either intersect at one point, or are
parallel.
\end{cor}

\begin{proof} Let $\g$ be a circle with center $O$. Since any line
through $O$ is normal to $\g$ we have $d_O=d_{n_{\g}}$. Hence, this
corollary follows from Theorem \ref{A}.
\end{proof}

To prove the above results, we first parametrized the goal function (the corresponding distance) by a real-valued parameter and then used the standard optimality condition: {\it the derivative of the goal function is zero at the points of local extremum}. This was possible, loosely speaking, because the set of all admissible triples/pairs was sufficiently rich to ensure that such a parametrization is possible and the optimality condition is valid. However, the phenomenon {\it the three normals are either concurrent or parallel} sometimes remains valid also when the set of admissible triples/pairs is not rich at all. Here is an example.

\begin{exam}\label{ex3} Consider the curves $\g_1: y=(x-1)^2, x>0$, $\g_2: y=\frac{(x+2)^2}{2}, x<0 $, and the point  $O=(0,0)$ (Fig.1).
\end{exam}

\begin{figure}[h]
  \centering
  \includegraphics[width=6cm]{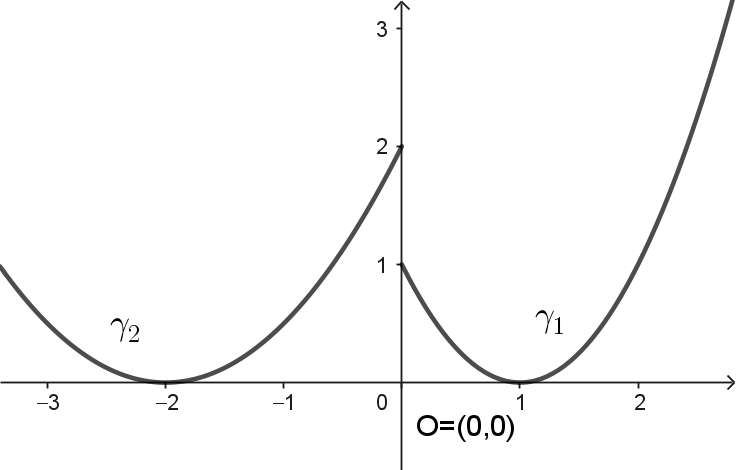}
  \caption*{Fig.1}
\end{figure}

Note that if
$A_1\in \g_1 \ , \ A_2\in \g_2$ and $O\in A_1A_2$, then
$A_1=(1,0)$ and $A_2=(-2,0)$. In this case, there is only one admissible pair
$(A_1, A_2)$, and $d_O(A_1, A_2)=3$. The three normals in this case are, nevertheless, parallel.
\medskip

The considerations till now may leave the wrong impression that the second half of the phenomenon ( {\it ``the three normals are parallel''}) appears only in exceptional and rather degenerate cases. Here is a ``regular'' example (with the function $d_O(A_1, A_2)$ to be minimized) in which the three normals are also parallel.

\begin{exam}\label{ex4}
 Let $\g_1$ be the curve
$$
\g_1(x) = \left\{ \begin{array}{ll}

(x,-(x-1)^2 )\ , \ 0<x\leq 1\\[5pt]
(x,(x-1)^2 )\ ,  \ x\geq 1.
\end{array}\right.
$$
Denote by $\g_2$ the curve symmetric to $\g_1$ with respect to the ordinate axis $Oy$ (Fig.2).
\end{exam}

\begin{figure}[h]
  \centering
  \includegraphics[width=6cm]{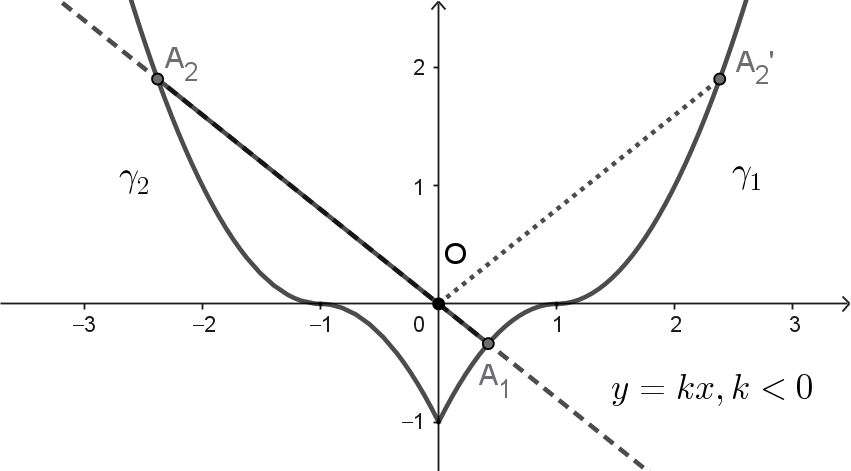}
 \caption*{Fig.2}
\end{figure}

To calculate the minimal value of the function
$d_O(A_1, A_2)$, where $O=(0,0)$,
let $A_1$ and $A_2$ be the intersection points of the line $y=kx$, $k \in \Bbb R$, with the curves
$\g_1$ and $\g_2$ correspondingly.
An easy calculation shows that for every $k \in \Bbb R$
$$d_O(A_1, A_2) =(2 +|k|)\sqrt{1 +k^2}.$$
Thus, the minimal distance between $A_1$ and $A_2$ is 2 and it is attained for $k=0$.  The pair $(A_1(1,0), A_2(-1, 0))$ is the only global minimum  for $d_O$. The three normals (to $\g_1$ at $A_1$, to $\g_2$ at $A_2$ and to the line $A_1A_2$ at $O$) are parallel to the ordinate axes $Oy$.

To reach this conclusion, we in fact  parametrized the curves
$\g_1$ and $\g_2$ by the parameter $k$ in such a ``coordinated'' way that, if $A_1(k) =\g_1(k)$, then the line $A_1(k)O$ intersects $\g_2$ at $A_2(k)=\g_2(k)$. While serving our immediate goals well, this parametrization
has a serious drawback. As seen from the formula

$$d_O(A_1(k), A_2(k))=(2 +|k|)\sqrt{1 +k^2},$$

\noindent
this parametrization is not differentiable at $k=0$. In all our considerations till now, however, differentiability played an important role. To get rid of this drawback, we provide  next a smooth parametrization of the curves
$\g_i, i=1,2$.  Let $\g_1$ be parametrized by the parameter $t=x$ with $t>0$. For the curve $\g_2$ we introduce a parametrization as follows:
$$
\g_2(t) = \left\{ \begin{array}{ll}

(-\ds \frac{1}{t},(1-\ds \frac{1}{t})^2) \ , \ 0<t\leq 1\\[5pt]
(-\ds \frac{1}{t},-(1-\ds \frac{1}{t})^2) \ , \ t\geq 1.
\end{array}\right.
$$
This parametrization is smooth and
also ``coordinated'': if $A_1(t)=\g_1(t) , t> 0$,  then the line $A_1(t)O$ intersects
$\g_2$ at the point $A_2(t)=\g_2(t)$. In this case
$$d_O^2(A_1(t), A_2(t))=(t+\frac{1}{t})^2((t+\frac{1}{t})^2-4(t+\frac{1}{t}) +5).$$
We already know that the global minimum of the function $d_O(A_1(t),A_2(t))$ is attained at $t=1$ since $g_1(1)=A_1(1)=(1,0)$ and $\g_2(1) = A_2(1)=(-1,0)$.

\medskip

\begin{cor}\label{cor3}  Let $\g_1$ be a $C^1$-smooth
curve and $O \notin \g_1$ be a fixed point. If
the distance from $O$ to a point of $\g_1$, has a local extremum
at $A_1\in g_1$, then $OA_1$ is the normal to $\g_1$ at $A_1$.
\end{cor}

\begin{proof}
This well-known result is a direct consequence of
Corollary \ref{cor1} taking $\g_2$ to be any $C^1$-smooth curve
through $O$ which is not tangent to $OA_1$ at $O$.
\end{proof}

Now we will analyze in more detail the case $O=A_1$ in Corollary
\ref{cor2} .

\begin{prop}\label{alpha} Let $\g_1$ be a $C^{1,1}$-smooth
curve, $\g_2$ be a $C^1$-smooth curve and $O\in \g_1$. Then
$(O,A_2)$ is a local extremum for the function $d_O$ if and
only if $A_2$
is a point of local extremum for the function $d(A)=|OA|, A\in\g_2$, and the line $OA_2$ is not tangent to $\g_1$ at $O$.
\end{prop}

\begin{proof} Let $(O,A_2)$ be a point of local extremum for the function
$d_O$. Then it is a point of local extremum  for the function $d$,
the distance from $O$ to a point of $\g_2$, and it follows from
Corollary \ref{cor3} that  $OA_2$ is the normal to $\g_2$ at
$A_2$.

To prove that the line $OA_2$ is not tangent to $\g_1$ at $O$
assume the contrary and consider an orthogonal coordinate system
$Oxy$ with positive $x$-axes $\overrightarrow{OA_2}$ and
$A_2=(1,0)$ . Then near $O$ the curve $\g_1$ is the graph of a
$C^{1,1}$-smooth function $f$ , i.e. $\g_1(x)=(x,f(x))$ with
$f(x)=O(\mid x\mid^2)$. Since $OA_2$ is the normal to $\g_2$ at
$A_2$ we can parameterize $\g_2$ near $A_2$ in the following way
$\g_2(t)=(1+g(t) , t)$,where $g(t)=o(\mid t\mid)$. Let
$A'_1=(x,f(x))$ for some $ x\neq 0$ . Then there is a unique point
$A'_2\in \g_2$ such that $O\in A'_1A'_2$. It is determined
uniquely by the identity
\begin{equation}\label{11}
\frac{f(x)}{x}=\frac{t}{1+g(t)}=h(t)
\end{equation}
since $h'(0)=1$, i.e. the function $h(t)$ is strictly increasing
near $0$. For any $x\neq 0$ near $0$ denote by $t(x)$ the unique
solution of (\ref{11}). Then $t(x)$ is a $C^1$-smooth function and
there are open neighborhoods $U_1$ of $A_1$ and $U_2$ of $A_2$
such that for any $A'_1=\g_1(x)\in U_1\setminus\{O\}, $ there is a
unique $A'_2=\g_2(t(x))\in U_2$ such that $O\in A'_1A'_2$. On the
other hand it follows from $f(x)=O(|x|^2)$ and $(\ref{11})$
that $t(x)=O(|x|)$ and therefore
$$d^2_O(A'_1,A'_2)=(1+g(t(x))-x)^2 +(t(x)-f(x))^2=1-2x+o(|x|)$$
for any $x\neq 0$ near $0$. Hence $(O,A_2)$ is not a point of local
extremum for $d_O$, a contradiction.

Conversely, suppose that the line $OA_2$ is not tangent to $\g_1$
at $O$. Then there are open neighborhoods $U_1$ of $O$ and $U_2$
of $A_2$ such that if $A'_1\in U_1, A'_2\in U_2$ and $O\in
A'_1A'_2$, then $A'_1=O$. Hence in this case
$d_O(A'_1,A'_2)=|OA'_2| = d(A'_2)$ and $(O,A_2)$ is a point of local
extremum for the function $d_O$ if and only if $A_2$ is a point of
local extremum for the function $d$.
\end{proof}

\begin{rem}\label{rem1} Suppose that the curves $\g_1$ and $\g_2$ are
respectively $C^{1,\alpha}$ and $C^{1,\beta}$-smooth, where $0\leq
\alpha, \beta \leq1 (C^{1,0}=C^1)$. Similar reasoning as above
shows that Proposition \ref{alpha}, where $\alpha=1, \beta=0$,
remains also true if $\alpha(1+\beta)>1$. To see what happens in
the case $\alpha(1+\beta)\leq 1, \beta \in (0,1] $ consider the
curves
$$\g_1:y=C|x|^{1+\alpha},\ \g_2:x=1+|y|^{1+\beta} , $$
where $C$ is a constant, and the points $A_1=O=(0,0) , A_2=(1,0)$.
Then $A_2$ is a point of global minimum of the function $d(A)=
|OA|, A\in \g_2$. If $\alpha(1+\beta)< 1$, then for any constant
$C$ the pair $(O,A_2)$ is not a point of local extremum  for $d_O$,
while for $\alpha(1+\beta)=1$ it is such a point if and only if
$\beta \in (0,1) , |C|\geq 1$ or $\beta=1,
|C|\ge\sqrt{\frac{2}{3}}$. Note that in all these cases $(O,A_2)$
is a point of local minimum for the function $d_O$.
\end{rem}
\smallskip

We next consider the points of local extremum for the function $d_{\g}$.

\begin{prop} \label{norm3} Let $\g_1 , \g_2$ and $\g$ be
$C^1$-smooth curves. Consider all pairs
$(A_1,A_2)$ such that $A_1\in \g_1 , A_2\in \g_2$
and $A_1A_2 \cap \g \neq \emptyset$. If
$(A_1,A_2)$ is a point of local
extremum for the function $d_{\g}$ then either the line $A_1A_2$ is
tangent to $\g$ or $A_1A_2$ is a common
normal to $\g_1$ and $\g_2$ at $A_1$ and $A_2$
correspondingly. Moreover, in both cases, if $A\in A_1A_2 \cap \g$, then the
normals $n_{A_1} , n_{A_2}$ and $n_A$ either intersect at one
point, or are parallel.
\end{prop}

\begin{proof} Note first that for any $A\in A_1A_2 \cap \g$ the
pair $(A_1,A_2)$ is a point of local extremum for the function
$d_A$. Hence it follows from Corollary \ref{cor2} that the normals
$n_{A_1}, n_{A_2}$ and the line through $A$ and perpendicular to
$A_1A_2$ either intersect at one point or are parallel. If
$A_1A_2$ is tangent to $\g$ at $A$, we are done. If it is not
tangent to $\g$ at $A$, then $(A_1,A_2)$ is a point of local extremum
for the the distance $d$ from $A_1$ to the
curve $\g_2$. This is so, because in this case for any point
$A'_2\in \g_2$ near $A_2$ the line $A_1A'_2$ intersects $\g$.
Hence it follows from Corollary \ref{cor3} that $A_1A_2$ is the
normal to $\g_2$ at $A_2$. The same reasoning shows that $A_1A_2$
is the normal to $\g_1$ at $A_1$. In this case the three normals meet at the point $A$ for any $A\in A_1A_2 \cap \g$ for trivial reasons. This completes the proof of the proposition.
\end{proof}

Finally, we consider the points of local extremum for the function
$d_{t_{\g}}$.

\begin{thm}\label{B} Let  $\g_1$ and $\g_2$ be $C^1$-smooth curves,
and $\g$ be a $C^2$-smooth curve with signed curvature
$\kappa_{\g}$. Consider all pairs $(A_1,A_2)$ such that
$A_1\neq A_2$ and the line $A_1A_2$ is tangent to $\g$ at some point
$A$. If $\kappa_{\g}(A)\neq 0$ and the function $d_{t_{\g}}$ has a
local extremum at $(A_1,A_2)$, then the normals
$n_{A_1},n_{A_2}$ and $n_A$ either intersect at one point, or are
parallel.
\end{thm}

\begin{proof} Let $s$ be the arc-length parameter of $\g$ with
$A=\g(0)$. Denote by $\widetilde{\g}$ the involute of $\g$ defined
by $\widetilde{\g}(s)= \g(s)-(s+a)\g\,'(s)$, where $a\neq 0$. We
will first prove that the curve $\widetilde{\g}$ is $C^2$-smooth
near $ \widetilde{A}=\widetilde{\g}(0)$. To do this, set
$\g(s)=(p(s), q(s))$ and $\widetilde{\g}(s)=(\widetilde{p}(s),
\widetilde{q}(s))$. Then $\widetilde{p}'(s)=-(s+a)p\,''(s) \  , \
\widetilde{q}'(s)=-(s+a)q\,''(s)$. We may assume that $q\,'(0)\neq 0$
Then the identity  $p\,'p\,''+q\,'q\,''=0$ and
$\kappa_{\g}(A)\neq 0$ imply that $p\,''(0)\neq 0$ and we get
from above that, for $s$ near $0,$
$$\frac{\widetilde{q}'(s)}{\widetilde{p}'(s)}=
\frac{q\,''(s)}{p\,''(s)}=-\frac{p\,'(s)}{q\,'(s)}.$$ Hence the function
$f=\displaystyle\frac{\widetilde{q}'}{\widetilde{p}'}$ is
$C^1$-smooth with respect to $s$ and therefore it is also
$C^1$-smooth with respect to the arc-length parameter $\widetilde
s$ of $\widetilde{\g}$. Since $(\widetilde p\,'(\widetilde
s))^2=\displaystyle\frac{1} {1+f^2(\widetilde s)}$ and $\widetilde
p\,'(0)\neq 0,$ we conclude that $\widetilde p\,'$ is $C^1$-smooth,
i.e. $\widetilde p$ is $C^2$-smooth w.r.t. $\tilde s$ near $0.$
The same is true for $\widetilde q$ as well.

We are now ready to prove the theorem. It is well-known that
$\kappa_{\widetilde{\g}}(\widetilde{A})=\displaystyle\frac{1}{a}\neq
0$ and $A$ is the center of curvature $C_{\widetilde{A}}$ of
$\widetilde{\g}$ for $\widetilde{A}$. Note also that the normal to
$\widetilde{\g} $ at a point $\widetilde{\g}(s)$ is the tangent of
$\g$ at $\g(s)$. Hence Theorem \ref{B} follows from Theorem \ref{A}.
\end{proof}

\begin{rem}\label{rem2}
(i) Theorem \ref{A} and  Theorem \ref{B} are in a sense equivalent
since Theorem \ref{B} (resp. Theorem \ref{A} ) is a consequence of
Theorem \ref{A} (resp. Theorem \ref{B}) applied to an involute
$\widetilde{\g}$ (the evolute $\widehat{\g}$) of the $C^2$-smooth
curve $\g$. In the first case we proved that $\widetilde{\g}$ is
$C^2$-smooth, whereas in the second case, the $C^2$-smoothness of
$\widehat{\g}$ can be ensured only if $\g$ is $C^3$-smooth and
$\kappa'_{\g}\neq 0.$ More precisely, one can prove that if a
curve $\g$ is $C^2$-smooth and $\kappa_{\g}(A)\neq 0$ for some
$A\in \g$, then its evolute $\widehat{\g}$ is $C^2$-smooth near
$C_A$ if and only if $\g$ is $C^3$-smooth near $A$ and
$\kappa'_{\g}(A)\neq 0$. Note that in this case, we have the
following formula for the signed curvature of
$\widehat{\g}$:

\begin{equation}\label{15}
\kappa_{\widehat{\g}}(C_A)=-\frac{\kappa^3_{\g}(A)}{\kappa'_{\g}(A)}.
\end{equation}

(ii) Theorem \ref{B} can be proved also by using similar arguments
as in the proof of Theorem \ref{A}. The key point in this approach
is the following analog of formulas (\ref{D1}) and (\ref{D2}) :
$$\overrightarrow{AB_i}= \frac{(1+d'_i(0))N}{\kappa_{\g}(A)}, i=1,2,$$
where $B_i=n_{A_i}\cap n_A.$
\

(iii) Consider the case $\kappa_{\g}(A)=0$ in Theorem \ref{B} when
$A\neq A_1,A_2$ and the line $A_1A_2$ is not tangent to $\g_1$ and
$\g_2$ at $A_1$ and $A_2$, respectively. Then writing $d_{t_{\g}}$
as a function of the arc-length parameter $s$ of $\g$ we see that
$d'_{\g}(0)=0$, i.e. $(A_1,A_2)$ is a critical point of
$d_{t_\g}$. Examples show that it can be a point of local minimum,
a point of local maximum, or not be a point of local extremum of
$d_{t_{\g}}$.
\end{rem}
\smallskip

Finally,  we consider the case $\g_2=\g$, when the function
$d_{t_{\g}}$ is defined for pairs of points $(A_1,A)$, such that
$A_1\in \g_1, A\in \g , A_1\neq A$ and the line $A_1A$ is tangent
to $\g$ at $A$.

\begin{cor}\label{cor4} Let $\g_1$ be a $C^1$-smooth curve and $\g$ be a
$C^2$-smooth curve with signed curvature $\kappa_{\g}$. Suppose
that the function $d_{t_{\g}}$ has a local extremum at $(A_1,A)$,
where $A_1\in \g_1, A\in \g , A_1\neq A$ and the line $A_1A$ is
tangent to $\g$ at $A$.

(i)  If $\kappa_{\g}(A)\neq 0$, then $n_A \cap n_{A_1}=C_A$,
where $C_A$ is the  center of curvature of $\g$ for $A$.

(ii) If $\kappa_{\g}(A)=0$, then $n_A\parallel n_{A_1}$, i.e
$A_1A$ is the tangent to $\g_1$ at $A_1$.
\end{cor}

\begin{proof} (i) If $A_1=C_A,$ it is obvious that $n_A \cap
n_{A_1}=C_A$. If $A_1\neq C_A,$ we may apply Corollary \ref{cor1}
to the involute of $\g$.

(ii) Suppose that $\kappa_{\g}(A)=0$ and
$A_1A$ is not the tangent to $\g_1$ at $A_1$.
Let $\g(s)=(p(s), q(s)), s \in \Delta$, and
$\g_1(t)=(p_1(t), q_1(t)), t\in \Delta_1$. Put
$$F(t,s) = -q\,'(s)(p_1(t) - p(s)) + p\,'(s)(q_1(t) - q(s)).$$
\noindent Clearly, $F(t,s)=0$ if and only if the line $\g(s)\g_1(t)$ is tangent to  $\g$ at $\g(s)$.
Reasoning as in Lemma \ref{l1} and Proposition
\ref{nnorm1} (with the function $G$ replaced by the function $F$) we derive that there exist an open interval $U \ni 0$ and a
$C^1$-smooth function $t: U \rightarrow \Delta_1$ such that $A_1=\g_1(t(0))$, and the line
$\g(s)\g_1(t(s))$ is tangent to $\g$ at the point
$\g(s)$ for every $s\in U$. Hence
$\g_1(t(s))= \g(s)+ d(s)T(s)$, where
$d(s)=\langle \g_1(t(s)) - \g(s), T(s)\rangle$ is the oriented distance from $A(s)=\g(s)$
to $A_1(s)=\g_1(t(s))$ with respect to the direction given by $T(s)$.
Since $A\neq A_1$, $s=0$ is a local extremum for $d(s)$. Hence, $d'(0)=0$. Note also that
$T'(0)=\kappa_{\g}(A)N(0)=0$ according to the Frenet-Serret formulas. Thus for the value of the derivative of $\g_1(t(s))$ at  $s=0$ we have

$$\g\,_1(t(0)) t'(0) =
\g\,'(0)+d'(0)T(0)+d(0)T'(0)=\g\,'(0).$$

\noindent I.e. the tangent vectors of $\g_1(t(s))$ and $\g(s)$  at $s=0$ are collinear. It follows that $A_1A$ is the tangent to $\g_1$
at $A_1$, a contradiction.
\end{proof}

\section{Appendix}

Let $\g(s), s\in \Delta$,  be a $C^2$-smooth curve parametrized by the arc-lenght parameter $s$.
If the curvature of $\g$ is not vanishing, every small connected segment of $\g$ can be viewed as a graph of a strictly convex (or strictly concave) function.
In this case, the tangential angle $\theta$ (the angle between the tangent at $\g(s)$ and the tangent at some fixed point, say $\g(0)$), is a strictly monotone function of $s$ which establishes a one-to-one correspondence between $s$ and $\theta$: $s=s(\theta)$ and $\theta = \theta(s)$, where $s(\theta)$ and $\theta(s)$ are smooth functions. The relation between $s$ and
$\theta$ is expressed through the Whewell intrinsic equation of a curve
\cite{WW}
$$\frac{ds(\theta)}{d\theta} =
\frac{1}{\kappa_{\g}(s)}= R_{\g}(s)$$
\noindent where $R_{\g}(s)$ is the signed radius of curvature of $\g$ at the point $\g(s)$.
Because of the monotonicity of $s$ and $\theta$ this equation can be rewritten  as
$$\frac{d\theta(s)}{ds} = \kappa_{\g}(s).$$

The considerations in the previous sections can be rearranged by taking
$\theta$ instead of $s$ as a natural parameter for the curve $\g$. If this is done, some of the formulas become simpler, and their interpretation gains geometrical clarity. For instance,   equation (\ref{7})
can be rewritten to become
$$\langle(\g_1(t_0) - C_{\g(s_0)}), N(s_0)\rangle=
\frac{G'_s(t_0,s_0)}{\kappa_{\g}(s_0)}= \frac{\partial G(t_0,s_0)}{\partial s}\, . \, \frac{ds(\theta)}{d\theta}=\frac{\partial G(t_0,s_0)}{\partial \theta}.$$
\noindent This offers the obvious interpretation that the partial derivative of
$G(t, s(\theta))$ with respect to
$\theta$ at the point $(t_0,s(\theta_0))$ as the signed distance from
$\g_1(t_0)$ to the center of curvature of $\g$ at the point $\g(s(\theta_0))$.

\noindent Equation (\ref{D1})
 gets the form (for $s=0$)
$$\overrightarrow{C_AB_1} =-\frac{\frac {d(d_1(s))}{ds}}{\frac{d\theta}{ds}}T =-\frac{d(d_1(s))}{d\theta}T .$$
 Therefore, the vector $\overrightarrow{C_AB_1}$ can be considered as a geometric expression of the derivative of $- d_1(s)$ with respect to $\theta$ at the point $s=0$. Similarly, the vector
 $\overrightarrow{C_AB_2}$ can be considered as a geometric expression of the derivative of $- d_2(s)$ with respect to $\theta$ at the point $s=0$.  Accordingly, the vector $\overrightarrow{B_1B_2}$ is the geometrical equivalent of the derivative of the function
$d(s)=d_2(s)-d_1(s)$ with respect to $\theta$ at the point $s=0$,   and this derivative vanishes just when $B_1 =B_2$ (see Figure 3 where the dashed lines are the normals).

\begin{figure}[h]
  \centering
  \includegraphics[width=6cm]{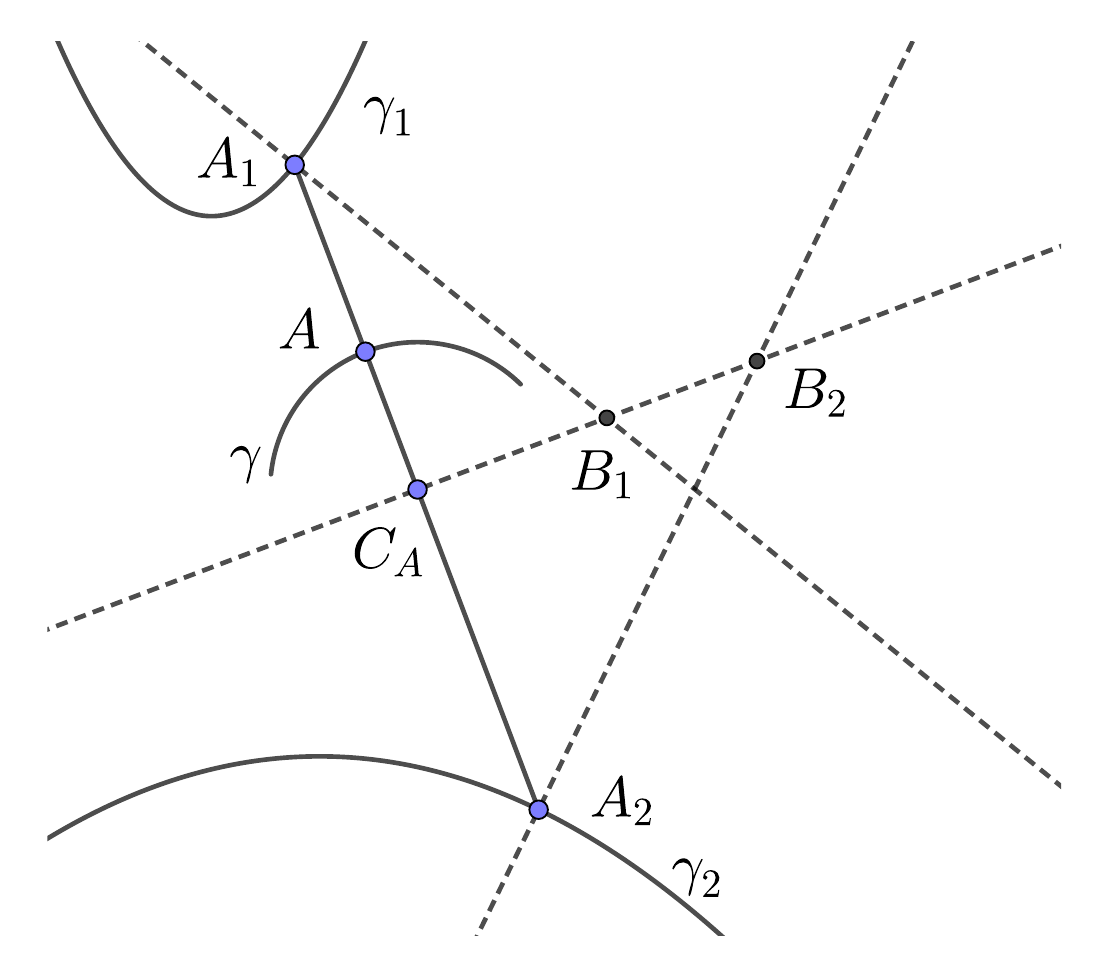}
  \caption*{Fig. 3}
\end{figure}
\vspace{5cm}

Further, since
$$\kappa_{\g}'(s) =\frac{ d \kappa_{\g}(s) }{ ds }=
- \frac{  1}{ R_{\g}^2(s) }.\frac{d R_{\g}(s) }{ds} =
- \kappa_{\g}^2(s)\frac{ dR_{\g}(s) }{ds},$$

equation (\ref{15}) gets the form

$$\frac{1}{R_{\widehat{\g}}(C_{\g(s)}) }=\kappa_{\widehat{\g}}(C_{\g(s)})= -\frac{\kappa^3_{\g}(s)}{\kappa'_{\g}(s)} =
\frac{\kappa_{\g}(s)}{ \frac{ dR_{\g}(s) }{ds} } =\frac{ 1}{ \frac{ dR_{\g}(s) }{d\theta} }.$$
Thus, equation (\ref{15}) is equivalent to the known (see \cite{WW}, p. 660) and
easy-to-remember formula relating the radii of curvature of the curves $\g$ and
$\widehat{\g}$:
$$R_{\widehat{\g}}(C_{\g(s)})= \frac{ dR_{\g}(s) }{d\theta}.$$


\end{document}